\documentclass[a4paper,11pt]{article}
\usepackage{setspace}
\usepackage{diagrams}
\usepackage{geometry}
\usepackage{epsfig}
\usepackage{amsfonts}
\usepackage{amsthm}
\usepackage{latexsym}
\usepackage{amsmath}
\usepackage{amscd}
\usepackage{amssymb}
\usepackage{enumerate}
\usepackage{graphicx,oldgerm}

\theoremstyle{definition}
\newtheorem{thm}{Theorem}[section]
\newtheorem{cor}[thm]{Corollary}
\newtheorem{lem}[thm]{Lemma}

\newtheorem{prop}[thm]{Proposition}
\newtheorem{defi}[thm]{Definition}
\newtheorem{rem}[thm]{Remark}
\newtheorem{note}[thm]{Notation}
\newtheorem{para}[thm]{}

\DeclareMathOperator{\codim}{\mathrm{codim}}
\DeclareMathOperator{\NL}{\mathrm{NL}}
\DeclareMathOperator{\prim}{\mathrm{prim}}

\DeclareMathOperator{\red}{\mathrm{red}}

\DeclareMathOperator{\Hc}{\mathcal{H}om}

\DeclareMathOperator{\Ima}{\mathrm{Im}}

\DeclareMathOperator{\op}{\mathcal{O}_{\mathbb{P}^3}(d)}

\DeclareMathOperator{\p3}{\mathbb{P}^3}

\DeclareMathOperator{\pr}{\mathrm{pr}}

\DeclareMathOperator{\Spec}{\mathrm{Spec}}

\DeclareMathOperator{\N}{\mathcal{N}}
\DeclareMathOperator{\T}{\mathcal{T}}

\DeclareMathOperator{\I}{\mathcal{I}}
\DeclareMathOperator{\mo}{\mathcal{O}}

\DeclareMathOperator{\coker}{\mathrm{coker}}

\newcommand{\mb}[1]{\mathbb{#1}}
\newcommand{\mc}[1]{\mathcal{#1}}
\newcommand{\mr}[1]{\mathrm{#1}}
\newcommand{\ov}[1]{\overline{#1}}
\title{On generically non-reduced components of Hilbert schemes of smooth curves}
\author{Ananyo Dan\thanks{The author has been supported by the DFG under Grant KL-$2244/2-1$\\ \newline Tata Institute for fundamental research, Homi Bhabha road, Mumbai-400005, Maharashtra, India
, e-mail: dan@math.tifr.res.in\\Mathematics Subject Classification ($2010$): $14$C$30$, $14$D$07$, $13$D$10$, $14$C$05$, $32$J$25$, $32$G$20$ }}
\date{\today}

\begin{document}
\maketitle
\begin{abstract}
 A classical example of Mumford gives a generically non-reduced component of the Hilbert scheme of smooth curves in $\p3$ such that a general element of the component is contained in a smooth cubic hypersurface in $\p3$. 
 In this article we use techniques from Hodge theory to give further examples of such (generically non-reduced) components of Hilbert schemes of smooth curves without any restriction on the degree of the hypersurface containing it.
 As a byproduct we also obtain generically non-reduced components of certain Hodge loci.
\end{abstract}

\section{Introduction}

In $1962$, only a few years after Grothendieck introduced the Hilbert scheme, Mumford \cite{Mu1} 
showed that there exists a generically non-reduced (in the sense, the localization of the structure sheaf at every point contains a non-trivial nilpotent element) irreducible component of the Hilbert scheme of curves in $\p3$
such that a general element of the component is a smooth curve contained in a cubic surface in $\p3$. This example inspired the investigation of such components.
Kleppe showed in \cite{kle2} that an irreducible component $L$ of the Hilbert scheme of curves parametrizing smooth curves contained in a cubic surface in $\p3$ is non-reduced if and 
only if for a general $C \in L$ and a smooth cubic surface $X$ containing $C$, $h^1(\mo_X(-C)(3)) \not=0$. Using this condition he gave examples in \cite{K1} of such non-reduced components.
Recently, Kleppe and Ottem in \cite{kle2012} have further generalized these techniques to study generically non-reduced components of Hilbert scheme of smooth curves contained in hypersurfaces in $\p3$ of degree at most $5$.

The approach by Mumford, Kleppe and Ottem is based on the study of flag Hilbert schemes. It is not clear whether their methods generalize to Hilbert schemes of curves not contained in smooth surfaces of small degree. 
One of the main problems is, these methods rely on a simple description of the map $H^1(\mo_X(C)) \to H^1(\N_{C|X})$ coming from the short exact sequence \[0 \to \mo_X \to \mo_X(C) \to \N_{C|X} \to 0.\]
In particular, in the case when $X$ is of degree $3$ or $4$, $h^1(\N_{C|X})=0$ (use Serre duality followed by the adjunction formula). Using this fact simplifies several computations coming from the study of flag Hilbert schemes (see \cite[\S $1.3$]{kle2}).
Unfortunately, such a simple description of the map from $H^1(\mo_X(C))$ to $H^1(\N_{C|X})$ is not available if $\deg(X)$ is greater than $5$.
In the case $\deg(X)=5$, they instead assume that the general element of the (generically non-reduced) component (of the Hilbert scheme of curves) satisfies the property: a general surface containing the curve has Picard number $2$ (see \cite{kle2012}).
Therefore, their methods cannot be used to study components of Hilbert schemes of curves whose general element is not contained in a surface of degree smaller than $5$ or if the Picard number of the surface is strictly greater than $2$.

We overcome the difficulty described above by combining analysis from flag Hilbert schemes with the theory of Hodge loci.
As an intermediate step we also obtain new generically non-reduced components of the Hodge locus (see Theorem \ref{a73}).
Owing to the linearity of the Gauss-Manin connection, it is easier to study deformation of generically non-reduced curves using this approach compared to simply using 
standard deformation theory.
The main result in this article is:
\begin{thm}\label{phe3}
For $d \ge 5$ and $m \ge 2d-2$, there exists a generically non-reduced irreducible component of the Hilbert scheme parametrizing smooth curves in 
$\p3$: 
\begin{enumerate}
 \item of degree $md+3$ and the arithemtic genus $1+(1/2)(md^2+d(m^2-4m-2)+6m+2)$,
 \item a generic element in this component corresponds to a smooth curve contained in a smooth degree $d$ surface in $\p3$ with Picard number (rank of N\'{e}ron Severi group)
 at least $3$ and not contained in any smooth surface of smaller degree. 
\end{enumerate}
\end{thm}
See Theorem \ref{mumsur01} for a more general statement and end of \S \ref{mg} for a proof of the above result. The proof does not depend on 
the Picard number of the surface containing it.

Heuristically, the strategy is to first show that for any $d \ge 5$ there exists a smooth degree $d$ surface $X$ and an effective divisor $D$ on $X$ such that the 
Hodge locus (inside the family of smooth degree $d$ surfaces in $\p3$) corresponding to the cohomology class of $D$ is generically non-reduced. 
The Hodge loci corresponding to semi-regular curves come from relevant flag Hilbert schemes.
We then show that there exists a smooth, semi-regular (in the sense of Bloch) curve $D'$ in the linear system $|D+mH|$ for $m \gg 0$ and $H$ a hyperplane section on $X$.
Since the Hodge loci corresponding to $D$ and $D'$ coincide, this will give us a generically non-reduced flag Hilbert scheme.
Finally, the generically non-reduced flag Hilbert scheme gives rise to a generically non-reduced component of the Hilbert scheme of smooth curves in $\p3$.

We now discuss the details of the strategy. Fix any degree $d \ge 5$. Denote by $U_d$ the space parametrizing smooth degree $d$ hypersurfaces in $\p3$. 
By a \emph{curve} we will always mean an effective divisor contained in a smooth hypersurface in $\p3$. To study non-reduced components of Hodge locus and 
Hilbert schemes of curves, the obvious first step is to relate the tangent space of both objects. This is done in Theorem \ref{hf12}, Corollaries \ref{hf12c} and \ref{dim4}.
Next, let us consider a generically non-reduced component $L$ of the Hilbert scheme of curves in $\p3$. Suppose a general element $C$ of $L$ is contained 
in a smooth degree $d$ surface, say $X$. Consider the Hodge class $[C]$ of this curve. We prove that under certain hypothesis on $C$, the corresponding Hodge locus 
is generically non-reduced (see Proposition \ref{a71}). The two main hypotheses for this to be possible are: $h^1(\mo_X(-C)(d))=0$ and as $X$ deforms along the Hodge locus, the cohomology 
class of $[C]$ deforms to the cohomology class of a deformation of $C$ (in the Hilbert scheme). The latter property is satisfied if $C$ is semiregular (see Theorem \ref{b71}).
Observe that the first property implies that for every infinitesimal deformation of $C$ there exists an infinitesimal deformation of $X$ containing it.
If $C$ is smooth then $C$ is $d$-regular for $d \ge \deg(C)$. This implies, if $C$ is smooth and $d>\deg(C)+4$ then $h^1(\mo_X(-C)(d))=0=h^1(\mo_X(C))$. Hence, such
$C$ is semiregular. In particular, if $C$ is smooth then it is not difficult to conclude the Hodge locus corresponding to $[C]$ is generically non-reduced. This includes
the examples coming from the generically non-reduced components of Hilbert schemes of smooth curves due to Mumford and Kleppe, discussed before.

Therefore, in this article, we do not elaborate much on the Hodge locus corresponding to cohomology classes of \emph{smooth} curves. The example we are interested in is a 
classical result of Martin-Deschamps and Perrin (see Theorem \ref{a84}). Here the generically non-reduced component of the Hilbert scheme 
parametrizes generically non-reduced curves, contained in smooth hypersurfaces in $\p3$. In this case, checking the two hypotheses mentioned above is non-trivial.
To satisfy the first hypothesis, we prove that given a smooth degree $d$ surface $X$ containing a line $l$ and a coplanar curve $C$, the Castelnuovo-Mumford regularity of 
the effective divisor $D:=2l+C$ is at most $d$ (see Theorem \ref{cas12}). The proof of the second hypothesis requires some computation (see proof of Theorem \ref{a73}).
We finally prove that:
\begin{thm}[See Theorem \ref{a73}]\label{phe4}
 Let $d \ge 5$, $X$ be a smooth degree $d$ surface in $\p3$ containing two coplanar lines $l_1, l_2$. Denote by $D:=2l_1+l_2$. The Hodge locus (in $U_d$) corresponding
 to the cohomology class $[D]$ of $D$ is generically non-reduced.
\end{thm}

The advantage of using the Hilbert scheme of curves in Theorem \ref{a84} is that the general element of the corresponding Hodge locus is a surface with Picard number at least $3$.
We use this component of the Hodge locus to prove Theorem \ref{phe3}. 

As mentioned earlier, if $C$ is a semiregular curve in a smooth degree $d$ surface in $\p3$ then the Hodge locus corresponding to $[C]$ parametrizes smooth degree $d$ 
surfaces $Y$ in $\p3$ such that $C$ deforms to a curve in $Y$. Now, we simply relate the Hodge locus in Theorem \ref{phe4} with the Hodge locus of a semiregular curve.
For this purpose, we take a general curve $D \in L$ ($L$ as in Theorem \ref{a84}) and a smooth degree $d$ surface $X$ containing $D$. For 
$m \gg 0$ and $H$ the hyperplane section in $X$, one uses Bertini's theorem to show that a general element, say $D'$ of the linear system $|D+mH|$ is smooth. Using
the definition of semiregularity, it is easy to show that $D'$ is semiregular. Since the cohomology class of the hyperplane section remains a Hodge class along $U_d$,
the Hodge locus corresponding to $[D]$ coincides with that of $[D']$. 
Finally, using standard theory of relative Hilbert schemes, we can conclude that the corresponding family of semiregular curves, with $D'$ as a special fiber, is generically non-reduced.

\begin{note}
 From now on a \emph{surface} will always mean a smooth surface in $\p3$ and a \emph{curve} will mean an effective divisor in a smooth surface. Note a curve need not be reduced.
 Also, given a scheme $Y$, we denote by $Y_{\red}$ the associated reduced scheme.
\end{note}

{\bf{Acknowledgement:}} I would like to thank my supervisor R. Kloosterman for introducing me to the topic, reading the preliminary version of this article and several helpful discussions.

\section{Preliminaries}

\subsection{Introduction to Hodge loci}

\begin{note}
Denote by $U_d \subseteq \mathbb{P}(H^0(\mathbb{P}^3, \op))$ 
the open subscheme parametrizing smooth projective hypersurfaces in $\mathbb{P}^3$ of degree $d$. Denote by $Q_d$ the Hilbert polynomial of degree $d$ surfaces in $\p3$. 
Given, a Hilbert polynomial $P$, denote by $H_P$ the corresponding Hilbert scheme and by $H_{P,Q_d}$ the corresponding flag Hilbert scheme.
Let $\mathcal{X} \xrightarrow{\pi} U_d$ be the corresponding universal family. For a given $F \in U_d$, denote by $X_F$ the surface $X_F:=\pi^{-1}(F)$. 
Let $X \in U_d$ and $U \subseteq U_d$ be a simply connected neighbourhood of $X$ in $U_d$ (under the analytic topology).
\end{note}

\begin{defi}
 As $U$ is simply connected, $\pi|_{\pi^{-1}(U)}$ induces a variation of Hodge structure $(\mathcal{H}, \nabla)$ on $U$ where $\mathcal{H}:=R^2\pi_*\mathbb{Z} \otimes 
\mathcal{O}_U$ and $\nabla$ is the Gauss-Manin connection. Note that $\mathcal{H}$ defines a local system on $U$ whose fiber over 
a point $F \in U$ is $H^2(X_F,\mathbb{Z})$. Consider a non-zero element $\gamma_0 \in H^2(X_F,\mathbb{Z}) \cap H^{1,1}(X_F,\mathbb{C})$
such that $\gamma_0  \not= c_1(\mathcal{O}_{X_F}(k))$ for $k \in \mathbb{Z}_{>0}$. 
This defines a section $\gamma \in (\mathcal{H} \otimes \mathbb{C})(U)$. Let $\overline{\gamma}$ be the image
of $\gamma$ in $\mathcal{H}/F^2(\mathcal{H} \otimes \mathbb{C})$. The \emph{Hodge loci}, denoted $\NL(\gamma)$ is then defined as
\[ \NL(\gamma):=\{G \in U | \overline{\gamma}_G=0\},\]
where $\overline{\gamma}_G$ denotes the value at $G$ of the section $\overline{\gamma}$. 
\end{defi}

\begin{note}\label{nl1}
 Let $X \in U_d$ and $\mo_X(1)$, the very ample line bundle on $X$ determined by the closed immersion $X \hookrightarrow \p3$ arising (as in \cite[II.Ex.$2.14$(b)]{R1}) from the graded homomorphism
of graded rings $S \to S/(F_X)$, where $S=\Gamma_*(\mo_{\p3})$ and $F_X$ is the defining equations of $X$.
Denote by $H_X$ the very ample line bundle $\mo_X(1)$. 
\end{note}

     \begin{note}
     Let $X$ be a surface. Denote by $H^2(X,\mb{C})_{\prim}$, the primitive cohomology. There is a natural
      projection map from $H^{2}(X,\mb{C})$ to $H^{2}(X,\mb{C})_{\prim}$. For $\gamma \in H^{2}(X,\mb{C})$, denote by $\gamma_{\prim}$ the image of $\gamma$ under this morphism.
      Since the very ample line bundle $H_X$ remains of type $(1,1)$ in the family $\mc{X}$, we can conclude that
        $\gamma \in H^{1,1}(X)$ remains of type $(1,1)$ if and only if $\gamma_{\prim}$ remains of type $(1,1)$. In particular,
      $\NL(\gamma)=\NL(\gamma_{\prim})$.
     \end{note}

 \begin{lem}[{\cite[Lemma $5.13$]{v5}}]
 There is a natural analytic scheme structure on $\ov{\NL(\gamma)}$ (closure in $U_d$ under Zariski topology).
\end{lem}

\begin{note}
 Let $X_1$ be a projective scheme, $X_2 \subset X_1$, a closed subscheme. Denote by $\N_{X_2|X_1}$ the normal sheaf $\Hc_{X_1}(\I_{X_2/X_1},\mo_{X_2})$, where $\I_{X_2/X_1}$ is the ideal sheaf of $X_2$ in $X_1$.
\end{note}

\begin{defi} 
We now discuss the tangent space to the Hodge locus, $\NL(\gamma)$.
We know that the tangent space to $U$ at $X$, $T_XU$ is isomorphic to $H^0(\N_{X|\p3})$.
This is because $U$ is an open subscheme of the Hilbert scheme $H_{Q_d}$, the tangent space of which at the point $X$ is simply
$H^0(\N_{X|\p3})$. 
Given the variation of Hodge structure above, we have (by Griffith's transversality) the differential map:
\[\overline{\nabla}:H^{1,1}(X) \to \mathrm{Hom}(T_XU,H^2(X,\mathcal{O}_X))\]induced by the Gauss-Manin connection.
Given $\gamma \in H^{1,1}(X)$ this induces a morphism, denoted $\overline{\nabla}(\gamma)$ from $T_XU$ to $H^2(\mo_X)$.
\end{defi}

\begin{lem}[{\cite[Lemma $5.16$]{v5}}]\label{tan1}
The tangent space at $X$ to $\NL(\gamma)$ equals $\ker(\overline{\nabla}(\gamma))$.
\end{lem}

\begin{defi}\label{gr0}
 The boundary map $\rho$, from $H^0(\N_{X|\p3})$ to $H^1(\T_X)$ arising from the long exact sequence associated to the 
 short exact sequence:
 \begin{equation}\label{ex9}
 0 \to \T_X \to \T_{\p3}|_X \to \N_{X|\p3} \to 0
 \end{equation}
  is called the \emph{Kodaira-Spencer} map. The morphism $\overline{\nabla}(\gamma)$ is related to the Kodaira-Spencer map as follows:
\end{defi}

\begin{thm}\label{a3}
The differential map $\overline{\nabla}(\gamma)$ conincides with the following:
\[T_XU\cong H^0(\N_{X|\p3}) \xrightarrow{\rho} H^1(\T_X) \xrightarrow{\bigcup \gamma} H^2(\mo_X)\]
and under the identification $\N_{X|\p3} \cong \mo_X(d)$, $\ker(\rho) \cong J_d(F)$, where $J_d(F)$ denotes the degree $d$ graded piece of the Jacobian ideal of $X$.
\end{thm}

\begin{proof}
 See \cite[Theorem $10.21$]{v4} and \cite[Lemma $6.15$]{v5} for a proof. 
\end{proof}

\subsection{Introduction to flag Hilbert schemes}

We briefly recall the basic definition of flag Hilbert schemes and its tangent space in the relevant case. See \cite[\S $4.5$]{S1}
for further details.

\begin{defi}
Given an $m$-tuple of numerical polynomials $\mathcal{P}(t)=(P_1(t),P_2(t),...,P_m(t))$, we define the  contravariant functor, called the \emph{Hilbert flag functor} 
relative to $\mathcal{P}(t)$,
\begin{eqnarray*}
FH_{\mathcal{P}(t)}:(\mbox{schemes}) &\to& \mbox{sets}\\
S &\mapsto& \left\{\begin{tabular}{l|l}
$(X_1,X_2,...,X_m)$& $X_1 \subset X_2 \subset ... \subset X_m \subset \mathbb{P}^3_S$\\
& $X_i \mbox{ are } S-\mbox{flat with Hilbert polynomial } P_i(t)$
\end{tabular}\right\}
\end{eqnarray*}
We call such an $m$-\emph{tuple a flag relative to} $\mathcal{P}(t)$. The functor is representable by a projective scheme $H_{\mc{P}(t)}$, called \emph{flag Hilbert scheme}.
\end{defi}

\begin{defi}\label{hf1}
In the case $m=2$, we have the following definition of the tangent space at a pair $(X_1,X_2)$ to the flag Hilbert scheme $H_{P_1,P_2}$:
\begin{equation}\label{dia2}
\begin{diagram}
&T_{(X_1,X_2)}H_{P_1,P_2} &\rTo &H^0(\N_{X_2|\mathbb{P}^3}) & \\
&\dTo & \square &\dTo & \\
&H^0(\N_{X_1|\mathbb{P}^3}) &\rTo &H^0(\N_{X_2|\mathbb{P}^3} \otimes \mathcal{O}_{X_1})
\end{diagram}
\end{equation} 
\end{defi}

\section{Semi-regularity map and Hodge locus}

\begin{para}
In this section, we consider the case when $\gamma$ is the cohomology class of a curve $C$ in a smooth degree $d$ surface $X$ in $\p3$ (recall $C$ need not be reduced). 
We see that the differential map $\ov{\nabla}(\gamma)$ factors through the semi-regularity map, $H^1(\N_{C|X}) \to H^2(\mo_X)$ (Theorem \ref{hf12}).
Using this description, we compute $\ker \ov{\nabla}(\gamma)$, which is the tangent space to the Hodge locus, $\NL(\gamma)$ at the point corresponding to $X$ (see Corollary \ref{hf12c}).
 \end{para}

\begin{defi}\label{mumsur02}
We start with the definition of a semi-regular curve.
 Let $X$ be a surface and $C \subset X$, a curve in $X$. Since $X$ is smooth, $C$ is local complete intersection in $X$.
Denote by $i$ the closed immersion of $C$ into $X$. This gives rise to the short exact sequence:
 \begin{equation}\label{sh1}
 0 \to \mo_X \to \mo_X(C) \to i_* \N_{C|X} \to 0.
                 \end{equation}                       
 The \emph{semi-regularity map} $\pi_C$ is the boundary map from $H^1(\N_{C|X})$ to $H^2(\mo_X)$.
 We say that $C$ is \emph{semi-regular} if $\pi_C$ is injective. 
\end{defi}



\begin{note}\label{hf23}
 Let $X$ be a smooth surface and $C$ a local complete intersection curve in $X$. There is a natural morphism $j_1:H^1(\T_X \otimes \mo_C) \to H^1(\N_{C|X})$ 
 induced by taking dual of the natural differential map, $d_X:\mo_X(-C)  \otimes \mo_C \to \Omega^1_X  \otimes \mo_C$.
 Let $u_*:H^1(\T_X) \to H^1(\N_{C|X})$ be the composition of the restriction morphism $j_3:H^1(\T_X) \to H^1(\T_X \otimes \mo_C)$
 with the morphism $j_1$ above.
\end{note}

\begin{thm}[{\cite[Theorem $4.5, 5.5$]{fle}}]\label{hf4}
Let $C, X$ be as in Notation \ref{hf23} and $[C]$ denote the cohomology class of $C$.
 The morphism $u_*$ defined above satisfies the following commutative diagram:
 \[\begin{diagram}
    H^1(\T_X)&\rTo^{u_*}&H^1(\N_{C|X})\\
    \dTo^{\bigcup [C]}&\ldTo^{\pi_C}\\
    H^2(\mo_X)
   \end{diagram}\]
   where $\bigcup [C]$ is the morphism which takes $\xi$ to $\xi \bigcup [C]$, the cup-product of $\xi$ and $[C]$.
\end{thm}

\begin{rem}
 One can see that $u_*$ is the obstruction map in the sense, for $\xi \in H^1(\T_X)$, $u_*(\xi)=0$ if and only if $C$ lifts to a local complete intersection in the infinitesimal deformation of $X$ corresponding to $\xi$ (see \cite[Proposition $2.6$]{b1}).
 We will now see a diagram (in Theorem \ref{hf12}) which illustrates this fact more clearly.
\end{rem}

\begin{para}
Recall, the following short exact sequence of normal sheaves:
 \begin{equation}\label{sh2}
  0 \to \N_{C|X} \to \N_{C|\p3} \to \N_{X|\p3}  \otimes \mo_C \to 0
 \end{equation}
 which arises from the short exact sequence:
 \begin{equation}\label{ex0c}
 0 \to \I_X \to \I_C \xrightarrow{j^\#} j_*\mo_X(-C) \to 0.
 \end{equation}
after applying $\Hc_{\p3}(-,j_{0_*} \mo_C)$, where $j_0$ (resp. $j$) is the closed immersion of $C$ (resp. $X$) into $\p3$.
\end{para}

The following lemma will be used in the proof of Theorem \ref{hf12} below. 

\begin{lem}\label{de1}
Given $j:C \hookrightarrow X$ the closed immersion, 
 the following diagrams are commutative and the horizontal rows are exact:
 \[\begin{diagram}
    0 &\rTo &\I_X  \otimes \mo_C &\rTo &\I_C  \otimes \mo_C &\rTo^u &\mo_X(-C) \otimes \mo_C &\rTo&0\\
    & &\dTo^{\cong}&\circlearrowright&\dInto^{d_{\p3}}&\circlearrowright&\dInto^{d_X}& \\
    0 &\rTo &\I_X  \otimes \mo_C &\rTo^{d_{\p3}} &\Omega^1_{{\p3}}  \otimes \mo_C &\rTo^v &\Omega^1_X  \otimes \mo_C &\rTo &0
   \end{diagram}\]
   where $d_{\p3}$ (resp. $d_X$) are the K\"{a}hler differential operators, $u$ is induced by $j^\#$ (as in (\ref{ex0c})) and $v$ 
   induced by the natural morphism $\mo_{\p3} \to j_*\mo_X$.
 \end{lem}

 \begin{proof}
The exactness of the lower horizontal sequence is explained in \cite[Theorem $25.2$]{mat}. 
We now show the exactness of the upper horizontal sequence. The sequence is obtained via pulling back (\ref{ex0c}) 
by the closed immersion of $C$ into $\p3$.
Since tensor product is right exact, it remains to prove that the morphim from $\I_X  \otimes \mo_C$ to $\I_C  \otimes \mo_C$
is injective. It suffices to prove this statement on the stalk  for all $x \in C$, closed point. Note that,
$\I_{X,x}  \otimes \mo_{C,x}$ (resp. $\I_{C,x}  \otimes \mo_{C,x}$) is isomorphic to $\I_{X,x}/\I_{X,x}\I_{C,x}$ (resp.
$\I_{C,x}/\I_{C,x}^2$). So, we need to show that $\I_{X,x} \cap \I_{C,x}^2$
is contained in $\I_{X,x}.\I_{C,x}$.

Since $C$ is a local complete intersection curve, $\I_{C,x}$ is generated by a regular sequence, say $(f_x, g_x)$ 
 where $g_x$ generates the ideal $\I_{X,x}$. So, $\I_{X,x}.\I_{C,x}=(g_x^2,f_xg_x), \I_{C,x}^2=(f_x^2,f_xg_x,g_x^2)$. Note that any element in $(g_x) \cap (f_x^2,f_xg_x,g_x^2)$ is divisible by $f_x^2$
modulo the ideal $(g_x^2,f_xg_x)$. Therefore it is divisible by $g_xf_x^2$, hence is an element in $\I_{X,x}.\I_{C,x}$. It directly follows that the natural morphism from 
$\I_{X,x}/\I_{X,x}.\I_{C,x}$ to $\I_{C,x}/\I_{C,x}^2$ is injective.


The commutativity of the diagrams follows directly from the description of the relevant morphisms.
\end{proof}


\begin{thm}\label{hf12}
Let $X$ be a smooth surface, $C \subset X$ and $\gamma=[C] \in H^{1,1}(X,\mb{Z})$.
 We then have the following commutative diagram
 \[\begin{diagram}
  & &  & &T_{(C,X)}H_{P,Q_d}&\rTo&H^0(X,\N_{X|\p3})&\rTo^{\ov{\nabla}(\gamma)}&H^2(X,\mo_X)\\
  & &  & &\dTo&\square&\dTo^{\rho_C}&\circlearrowright&\uTo^{\pi_C}& \\
    0&\rTo&H^0(C,\N_{C|X})&\rTo^{\phi_C}&H^0(C,\N_{C|\p3})&\rTo^{\beta_C}&H^0(C,\N_{X|\p3} \otimes \mo_C)&\rTo^{\delta_C}&H^1(C,\N_{C|X}) 
       \end{diagram}\]
where the horizontal exact sequence comes from (\ref{sh2}), $\pi_C$ is the semi-regularity map and $\rho_C$ is the natural
pull-back morphism.
\end{thm}

\begin{proof}
The only thing we need to prove is that $\ov{\nabla}(\gamma)$ is the same as $\pi_C \circ \delta_C \circ \rho_C$.
Using Theorems \ref{a3} and \ref{hf4}, we have that $\ov{\nabla}(\gamma)$ factors as:
\[H^0(\N_{X|\p3}) \xrightarrow{\rho} H^1(\T_X) \xrightarrow{u_*} H^1(\N_{C|X}) \xrightarrow{\pi_C} H^1(\mo_X).\]
Hence it suffices to show that $u_* \circ \rho$ is the same as $\delta_C \circ \rho_C$.

Recall,  under the notations as in \ref{hf23} we can factor $u_*$ as $j_1 \circ j_3$. Hence, it suffices to construct
a morphism $j_2:H^0(\N_{X|\p3} \otimes \mo_C) \to H^1(\T_X \otimes \mo_C)$ such that following two diagrams commute:
\begin{equation}\label{dc1}
\begin{diagram}
   H^0(\N_{X|\p3})&\rTo^{\rho}&H^1(\T_X)&\rTo^{j_3}&H^1(\T_X \otimes \mo_C)&\rTo^{j_1}&H^1(\N_{C|X})\\
   &\rdTo^{\rho_C} & &\ruDashto^{j_2}& &\ruTo(4,2)^{\delta_C}\\
   &  &H^0(\N_{X|\p3} \otimes \mo_C)
  \end{diagram}
  \end{equation}

We define $j_2$ in the following way: Take the following commutative diagram of short exact sequences:
\begin{equation}\label{com6}
\begin{diagram}
0 &\rTo &\T_X &\rTo &\T_{\p3} \otimes \mo_X &\rTo &\N_{X|\p3} &\rTo &0\\
& &\dTo&\circlearrowright&\dTo&\circlearrowright&\dTo\\
0 &\rTo &\T_X \otimes \mo_C &\rTo &\T_{\p3} \otimes \mo_C &\rTo &\N_{X|\p3} \otimes \mo_C&\rTo &0
\end{diagram}
\end{equation}
Then $j_2$ arises from the associated long exact sequence:
 \[\begin{diagram}
   H^0(\N_{X|\p3})& \rTo^{\rho} &H^1(X,\T_X)\\
   \dTo^{\rho_C}&\circlearrowright &\dTo^{j_3}\\
   H^0(\N_{X|\p3}  \otimes \mo_C)& \rTo^{j_2} &H^1(\T_X  \otimes \mo_C)
  \end{diagram}\]
  where the maps (other than $j_2$) is the same as defined above. This gives the commutativity of the first
  square in the diagram (\ref{dc1}).
  
  We now prove the commutativity of the second diagram.
Since the terms in the short exact sequences in Lemma \ref{de1} are locally free $\mo_C$-modules, we get the 
dual diagram of short exact sequence by applying $\Hc_C(-,\mo_C)$ to it. This gives us the following: 
\[\begin{diagram}
0 &\rTo &\N_{C|X} &\rTo &\N_{C|\p3} &\rTo &\N_{X|\p3} \otimes \mo_C &\rTo &0\\
& &\uTo&\circlearrowright&\uTo&\circlearrowright&\uTo^{\cong}\\
0 &\rTo &\T_X \otimes \mo_C &\rTo &\T_{\p3} \otimes \mo_C &\rTo &\N_{X|\p3} \otimes \mo_C&\rTo &0
\end{diagram}\]
where the bottom short exact sequence is the same as in (\ref{com6}).
Taking the associated long exact sequence,
we get 
 \[\begin{diagram}
    & &H^0(\N_{X|\p3} \otimes \mo_C)\\
    & \ldTo^{j_2} & \dTo^{\circlearrowright\, \, \,}_{\delta_C}\\
    H^1(\T_X \otimes \mo_C)&\rTo^{j_1}&H^1(\N_{C|X})
   \end{diagram}\]
The theorem follows.
\end{proof}

\begin{cor}\label{hf12c}
 Let $X$ be a smooth degree $d$ surface in $\p3$, $C \subset X$ be a curve with Hilbert polynomial, say $P$ and $[C]$ its corresponding cohomology class. 
Then, the tangent space, \[T_X(\NL([C])) \supseteq \rho_C^{-1}(\Ima \beta_C)= \pr_2T_{(C,X)}H_{P,Q_d}.\]
Furthermore, if $C$ is semi-regular then we have equality
$T_X(\NL([C]))=\pr_2T_{(C,X)}H_{P,Q_d}$.
\end{cor}

\begin{proof}
 The first statement follows directly from the diagram in Theorem \ref{hf12}.
Now, if $C$ is semi-regular then $\pi_C$ (as in the diagram) is injective. Hence, $\ker \overline{\nabla}([C])=\ker(\delta_C \circ \rho_C)=\rho_C^{-1}(\Ima \beta_C)$. The corollary then follows. 
\end{proof}

  \begin{cor}\label{dim4}
  Notations as in Theorem \ref{hf12}.
   The kernel of $\rho_C$ is isomorphic to $H^0(\mo_X(-C)(d))$ and $\rho_C$ is surjective if and only if $H^1(\mo_X(-C)(d))=0$.
   Moreover, if $H^1(\mo_X(-C)(d))=0$ then $\pr_1(T_{(C,X)}H_{P,Q_d})=H^0(\N_{C|\p3})$.
  \end{cor}

  \begin{proof}
   Since $\N_{X|\p3} \cong \mo_X(d)$ the first statement follows from the short exact sequence,
   \begin{equation}
 0 \to \mo_X(-C)(d) \to \mo_X(d) \to i_*\mo_C(d) \to 0
\end{equation}
and the fact that $H^1(\mo_X(d))=0$,
where $i$ is the closed immersion of $C$ into $X$. 

The last statement follows directly from the definition of $T_{(C,X)}H_{P,Q_d}$ given in \ref{hf1}.
     \end{proof}


\section{Non-reduced components of the Hodge locus}

In this section we give an explicit example of generically non-reduced Hodge locus (see Theorem \ref{a73}).

\subsection{General criterion for generically non-reduced components of the Hodge locus}

\begin{para}
 The aim of this subsection is to give a criterion for a Hodge locus to be generically non-reduced (see Proposition \ref{a71}). 
 This will be used in the proof of Theorem \ref{a73} to produce a concrete example of a generically non-reduced Hodge locus.
 \end{para}

Recall the following result,
\begin{lem}\label{a4e}
Let $d \ge 5$ and $C$ be an effective divisor on a smooth degree $d$ surface $X$ of the form $\sum_i a_iC_i$ where $C_i$ are integral curves with $\deg(C)+2 \le d$.
Then, $h^0(\N_{C|X})=0.$ In particular, $\dim |C|=0$ where $|C|$ is the linear system associated to $C$.
\end{lem}

\begin{proof}
See \cite[Lemma $3.6$]{D3}.
\end{proof}

\begin{lem}\label{dim3}
 Let $X$ be a smooth surface and $C$ a local complete intersection in $X$. Then, $h^0(\mo_X(-C)(d))=h^0(\I_C(d))-1$ and $h^0(\N_{C|X})=h^0(\mo_X(C))-1$.
\end{lem}

\begin{proof}
 The first equality follows from the short exact sequence,
 \[0 \to \I_X(d) \to \I_C(d) \to i_*\mo_X(-C)(d) \to 0\]
 and the fact that $\I_X \cong \mo_{\p3}$, where $i:X \to \p3$ is the natural closed immersion.
 
 The second equality follows from the short exact sequence (\ref{sh1}) after using the facts $h^0(\mo_X)=1$ and $h^1(\mo_X)=0$.
\end{proof}

We can then compute the dimension of $\ov{\NL(\gamma)}$ as follows:
\begin{lem}\label{dim}
Let $P$ be the Hilbert polynomial of a curve in $\p3$ and $L$ an irreducible component of the Hilbert scheme $H_P$. 
 Suppose there exists an integer $d \ge 5$ such that a general element of $L$ is contained in a smooth degree $d$ surface. Take a general pair $(C,X)$ where $C$ is an element in $L$ and $X$ is a smooth
 degree $d$ surface containing $C$. Denote by $\gamma=[C]$, the cohomology class of $C$. If there exists an irreducible component $H_\gamma$ in $(H_{{P,Q_d})_{\red}}$ such that 
 $\pr_2(H_\gamma)_{\red}=\ov{\NL(\gamma)}_{\red}$
The dimension of $\ov{\NL(\gamma)}$, \[\dim \ov{\NL(\gamma)}=\dim I_d(C)+\dim L-h^0(\mathcal{O}_X(C)).\]
\end{lem}

\begin{proof}
This follows from the fiber dimension theorem (see \cite[II Ex. $3.22$]{R1}) which states that for a morphism of finite type between two integral schemes, $f:X \to Y$,
$\dim X=\dim f^{-1}(y)+\dim Y$ for a general point $y \in Y$. 
We then have the following maps,
\[\pr_1:H_\gamma \twoheadrightarrow L\, \, \, \pr_2:H_\gamma \twoheadrightarrow \ov{\NL(\gamma)}\]
where the generic fiber of $\pr_1$ is $\mb{P}(I_d(C))$ and that of $\pr_2$ is $\mathbb{P}(H^0(\mathcal{O}_X(C)))$.
We then conclude, \[\dim {H}_\gamma=\dim L+\dim I_d(C)-1=\dim \ov{\NL(\gamma)}+h^0(\mo_X(C))-1.\]
Therefore,  $\dim \ov{\NL(\gamma)}=\dim L+\dim I_d(C)-h^0(\mathcal{O}_X(C))$.
\end{proof}

\begin{prop}\label{a71}
 Hypothesis as in Lemma \ref{dim}. If $L$ is a generically non-reduced component of the Hilbert scheme $H_P$ and $h^1(\mo_X(-C)(d))=0$ then  $\NL(\gamma)$ is generically non-reduced.
\end{prop}

\begin{proof}
As $h^1(\mo_X(-C)(d))=0$, Corollary \ref{dim4} implies $\rho_C$ is surjective. Then the diagram in \ref{hf1} tells us $\pr_1T_{(C,X)}H_{P,Q_d}=T_CL$.
By Corollary \ref{hf12c}, we can then conclude \[\dim T_X\NL(\gamma) \ge \dim \pr_2T_{(C,X)}H_{P,Q_d}=\dim T_{C,X}H_{P,Q_d}-\dim \ker \beta_C=\]
\[=\dim \pr_1T_{(C,X)}H_{P,Q_d}+\dim \ker \rho_C-\dim \ker \beta_C=\dim T_CL+\dim \ker \rho_C-\dim \ker \beta_C.\]
Using the normal short exact sequence, $\ker \beta_C \cong H^0(\N_{C|X})$. By Corollary \ref{dim4},  $\ker \rho_C \cong H^0(\mo_X(-C)(d))$.
Therefore, Lemma \ref{dim3} implies
\begin{eqnarray*}
 \dim T_X\NL(\gamma)&\ge& \dim T_{C}L+h^0(\mo_X(-C)(d))-h^0(\N_{C|X})\\
 &=&\dim T_CL+\dim I_d(C)-h^0(\mo_X(C)).
\end{eqnarray*}
Finally, by Lemma \ref{dim}, \[\dim T_X\NL(\gamma) - \dim \NL(\gamma) \ge \dim T_CL-\dim L.\]
So, $\NL(\gamma)$ is generically non-reduced if $L$ is generically non-reduced. This proves the proposition.
\end{proof}

\subsection{An Example due to Martin-Deschamps and Perrin}

\begin{para}
 In this subsection, we study some specific examples of effective divisor on smooth degree $d$ surfaces. We show that the Castelnuovo-Mumford regularity for the 
ideal sheaves of these divisors is at most $d$ (Theorem \ref{cas12}). This will imply that the curves satisfy the additional condition in Proposition \ref{a71} above.
\end{para}
 
\begin{note}\label{a7}
 Let $a, d$ be positive integers, $d \ge 5$ and $a>0$. Let $X$ be a smooth projective surface in $\p3$ of degree $d$ containing a line $l$ and a smooth coplanar curve $C_1$ of degree $a$.
 Let $C$ be a divisor of the form $2l+C_1$ in $X$. Denote by $P$ the Hilbert polynomial of $C$.
\end{note}

\begin{thm}[Martin-Deschamps and Perrin \cite{mar1}]\label{a84}
There exists an irreducible component, say $L$ of $H_P$ such that a general curve $D \in L$ is a divisor in a smooth degree $d$ surface in $\p3$ of the form $2l'+C_1'$ where $l', C_1'$ are coplanar curves with $\deg(l')=1$
and $\deg(C_1')=a$. Furthermore, $L$ is generically non-reduced.
\end{thm}

\begin{proof}
 The theorem follows from \cite[Proposition $0.6$, Theorems $2.4, 3.1$]{mar1}.
\end{proof}

\begin{note}\label{cas4}
Denote by $S$ the ring $\Gamma_*(\mo_{\p3})=\oplus_{n \in \mb{Z}} \Gamma(\p3,\mo_{\p3}(n))$.
 Let ${C_0}$ be a plane curve of degree $a+1$ for some positive integer $a$, defined by two equations, say $l_1$, $l_2G_1$ where $l_1, l_2$ are linear polynomials and $G_1$ is a smooth 
 polynomial of degree $a$. 
 Let $X$ be a smooth surface of degree $d$ containing ${C_0}$. Then $X$ is defined by an equation of the form $F_X:=F_1l_1+F_2l_2G_1$. By taking $X$ to be a general surface containing ${C_0}$, 
 we can assume that $G_1, F_1$ and $F_2$ are not contained
 in the ideal generated by $l_1, l_2$. Denote by $l$ the line defined by $l_1$ and $l_2$ and by $A_l$ its coordinate ring. 
 The aim of this section to prove that $\mo_X(-l) \otimes \mo_X(-{C_0})$ is $d$-regular.
 \end{note}
 
 We now recall a result on Castelnuovo-Mumford regularity which we will use later:
 \begin{thm}[{\cite[Ex. $4$E, Proposition $4.16$]{syz1}}]\label{syz1}
  Suppose that \[0 \to M' \to M \to M'' \to 0\]is a short exact sequence of finitely generated $S$-modules. Denote by $\mr{reg}(N)$ the Castelnuovo-Mumford regularity of a finitely generated $S$-module $N$.
  Then, $\mr{reg}(M') \le \max\{\mr{reg}(M),\mr{reg}(M'')-1\}, \mr{reg}(M'') \le \max\{\mr{reg}(M),\mr{reg}(M')+1\}, \mr{reg}(M) \le \max\{\mr{reg}(M'),\mr{reg}(M'')\}.$ Furthermore,
  given a finitely generated $S$-module $M$, $\widetilde{M}$ is $\mr{reg}(M)$-regular, where $\widetilde{M}$ is the associated coherent sheaf.  
 \end{thm}

 \begin{para}
 Denote by $I_{C_0}$ (resp. $I_X$, $I_{l}$) the ideal of ${C_0}$ (resp. $X, l$). 
 Then we have the following commutative diagram:
 \begin{equation}\label{com9}
\begin{diagram}
    0&\rTo&I_X&\rTo^{f_1}&I_{C_0}&\rTo^{f_2}&I_{C_0}/I_X&\rTo&0\\
    & &\uTo^{\cong}&\circlearrowleft&\uOnto^{f_4}& \\
    & &S(-d)&\rTo^{f_3}&S(-1) \oplus S(-a-1)& \\
    & & & &\uInto^{f_5}&\\ 
    & & & &S(-2-a)&\\
   \end{diagram}
   \end{equation}
where $f_1$ is the natural inclusion map, $f_2$ is the quotient map, \[f_3:G \mapsto (GF_1,GF_2),\, \, f_4:(H_1, H_2) \mapsto H_1l_1+H_2l_2G_1,\, \, f_5:G \mapsto (Gl_2G_1,-Gl_1).\]
Note that the top horizontal row and the rightmost vertical sequence are exact. 
\end{para}

\begin{lem}\label{cas6}
 The kernel of the natural morphism $f'_5+f'_3$ from   $A_{l}(-2-a) \oplus A_{l}(-d)$ to $A_{l}(-1)\oplus A_{l}(-a-1)$ which
 maps $(\ov{G},\ov{H})$ to $(\ov{HF_1+Gl_2G_1},\ov{HF_2-Gl_1})$ is $A_{l}(-a-2) \oplus 0$, where $G, H \in S(-a-1), S(-d)$, 
 respectively and $\ov{G},\ov{H}$ are  its images in $A_l(-2-a)$ and $A_l(-d)$, respectively
\end{lem}

\begin{proof}
 Since $I_{l}$ is generated by $l_1$ and $l_2$, $\ov{HF_1+Gl_2G_1}$ and $\ov{HF_2-Gl_1}$ are both 
 zero if and only if both $HF_1$ and $HF_2$ are both in $I_{l}$. Since $F_1, F_2$ are not contained in $I_{l}$, by assumption,
 and $I_{l}$ is a prime ideal, this implies $H$ is contained in $I_{l}$. In other words, the kernel of the map $f'_5+f'_3$ is isomorphic to pairs $(\ov{G},0)$ where $G \in S(-a-2)$. 
 This proves the lemma.
\end{proof}

\begin{prop}\label{cas8}
$A_{l} \otimes_S I_{C_0}/I_X$ is $d+1$-regular.
\end{prop}
 
 \begin{proof}
 Using the exactness of the sequences in the diagram (\ref{com9}), we can conclude that $\Ima f_3+\Ima f_5$ is contained in $\ker(f_2 \circ f_4)$. We now show the converse inclusion. Let $h \in \ker(f_2 \circ f_4)$.
 So, $f_4(h) \in \Ima f_1$, which implies that there exists $b \in S(-d)$ such that $f_4(h-f_3(b))=0$. Therefore, there exists $c \in S(-2-a)$ such that $f_5(c)=h-f_3(b)$. This proves that 
 $\ker(f_2 \circ f_4) \subset \Ima f_3+\Ima f_5$, hence equality.
 
 Note that the image of $f_5$ maps to zero under $f_4$ due to exactness of the sequence. But, $f_4 \circ f_3$ is injective by the commutative of the above diagram and the injectivity of $f_1$. 
 So, $\Ima f_3 \cap \Ima f_5=0$. So, $\ker(f_4 \circ f_2)=S(2-a) \oplus S(-d)$.
 So, we have the following short exact sequence of $S$-modules:
 \begin{equation}\label{eqcas1}
  0 \to S(-2-a) \oplus S(-d) \xrightarrow{f_5\oplus f_3} S(-1)\oplus S(-a-1) \xrightarrow{f_4 \circ f_2} I_{C_0}/I_X \to 0
  \end{equation}
  Since tensor product is right exact, using Lemma \ref{cas6} we get the following exact sequence:
  \[0 \to A_{l}(-d) \xrightarrow{f'_3} A_{l}(-1)\oplus A_{l}(-a-1) \xrightarrow{f'_4 \circ f'_2} A_{l} \otimes_S I_{C_0}/I_X \to 0\]
  where $f'_i$ are induced by $f_i$ for $i=1,...,5$. 
  
  Using the short exact sequence \[0 \to S(-2) \to S(-1) \oplus S(-1) \to S \to A_l \to 0\] we see that the Castelnuovo-Mumford
  regularity of $A_{l}(-d)$ and $A_{l}(-1)\oplus A_{l}(-a-1)$ is less than or equal to $d$.
  Finally, Theorem \ref{syz1} implies that the Castelnuovo-Mumford regularity of $A_{l} \otimes_S I_{C_0}/I_X$ is at 
  most $d+1$.
 \end{proof}

\begin{thm}\label{cas12}
 The sheaf $\mo_X(-l-{C_0})$ is $d$-regular. 
\end{thm}

\begin{proof}
 Consider the natural surjective morphim $A_X \to A_l$.
 Tensoring this by $I_{C_0}/I_X$, we obtain the short exact sequence,
 \[0 \to \ker(p) \to I_{C_0}/I_X \xrightarrow{p} A_{l} \otimes_S I_{C_0}/I_X \to 0.\]
 Using Proposition \ref{cas8} and Theorem \ref{syz1} we conclude that $\widetilde{\ker(p)}$ is $d$-regular. 
 It remains to prove that $\widetilde{\ker{p}} \cong \mo_X(-C_0-l)$.

 Since $\Gamma_*(\mo_X(-C_0))=I_{C_0}/I_X$ (by definition) and $\Gamma_* \mo_l=A_l$, \cite[Proposition II.$5.15$]{R1}
 implies that $\widetilde{I_{C_0}/I_X}$ and $\widetilde{A_l}$ is isomorphic to $\mo_X(-C_0)$ and $\mo_l$, respectively. 
 Now, the associated module functor $\widetilde{ }$  is exact and commutes with tensor product (\cite[Proposition II.$5.2$]{R1}).
 Applying this functor to the last short exact sequence we get,
 \[0 \to \widetilde{\ker(p)} \to \mo_X(-C_0) \xrightarrow{\widetilde{p}} \mo_l \otimes \mo_X(-C_0) \to 0\]
 where $\widetilde{p}$ arises from tensoring by $\mo_X(-C_0)$ the natural surjective morphism
 $\mo_X \to \mo_l$.
 
 Consider now the short exact sequence \[0 \to \mo_X(-l) \to \mo_X \to \mo_l \to 0.\]
 Since $\mo_X(-C_0)$ is a flat $\mo_X$-module, we get the short exact sequence,
 \[0 \to \mo_X(-C_0-l) \to \mo_X(-C_0) \xrightarrow{\widetilde{p}} \mo_l \otimes_{\mo_X} \mo_X(-C_0) \to 0.\]
 By the universal property of the kernel, $\widetilde{\ker{p}}$ is isomorphic to $\mo_X(-C_0-l)$.
 Hence, $\mo_X(-C_0-l)$ is $d$-regular.
\end{proof}

\subsection{Generically non-reduced component of Hodge locus}

\begin{para}
 We will see in the proof of Theorem \ref{a73} that the hypotheses of Lemma \ref{dim} are also satisfied for the generically non-reduced component $L$ mentioned in Theorem \ref{a84} above.
Proposition \ref{a71} then tells us the Hodge locus corresponding to a general curve in $L$ is generically non-reduced.
\end{para}

\begin{thm}\label{a73}
 Let $d \ge 5$, $X$ be a smooth degree $d$ surface containing $2$ coplanar lines, $l_1, l_2$. Let $C$ be a divisor in $X$ of the form $2l_1+l_2$, $\gamma$ the cohomology class of $C$.
 Then, $\ov{\NL(\gamma)}_{\red}$ is isomorphic to $\ov{(\NL([l_1]) \cap \NL([l_2]))}_{\red}$. In particular, $\ov{\NL(\gamma)}$ is irreducible, generically non-reduced.
\end{thm}

\begin{proof} 
 Note that $\NL(\gamma)$ contains the space of smooth degree $d$ surfaces containing two coplanar lines. But this space is of codimension $2d-6$. So,
 $\codim \ov{\NL(\gamma)} \le 2d-6$. If $\codim \ov{\NL(\gamma)} < 2d-6$ \cite[Proposition $1.1$]{v3} implies that $\ov{\NL(\gamma)}$ is reduced and either it parametrizes smooth degree $d$ 
 surfaces containing a line or containing a conic. This means that there exists $\gamma'$ a class of a line or a conic in $X$ such that $\NL(\gamma)=\NL(\gamma')$.
 Note that if $\gamma'_{\prim}$ is a multiple of $[l_1]_{\prim}$ or $[l_2]_{\prim}$ or $[l_1+l_2]_{\prim}$ then this locus parametrizes surfaces such that the cohomology class of $l_1$ and $l_2$ remains of type $(1,1)$.
 As $h^1(\mo_X(l_1))$ and $h^0(\mo_X(l_2))$ vanish, $l_1 \mbox{ and } l_2$ are semi-regular, for $d \ge 5$. Theorem \ref{dim1} then implies this space parametrizes surfaces containing $2$ coplanar lines, hence $\codim \ov{\NL(\gamma)} = 2d-6$, which gives us a contradiction.
 
 Choose $X$ in $\NL(\gamma)$ such that if $H$ is the hyperplane containing $l_1 \cup l_2$ then $H \cap X=l_1 \cup l_2 \cup D$ with $D$ irreducible. Note that such $X$ exists and the new $\gamma$ corresponding
 to the cohomology class of the divisor $2l_1+l_2$ in this surface defines the same component $\ov{\NL(\gamma)}$.
 Denote by $E$ the preimage in $H^0(\mo_{\p3}(d))$ of the vector space $T_X(\NL(\gamma)) \subset H^0(\mo_X(d))$. For an integer $k$, denote by \[E_k:=[E:S_{d-k}]:=\{P \in S_k|P.S_{d-k} \subset E\},\] where 
 $S_{d-k}$ is $H^0(\mo_{\p3}(d-k))$. Now, \cite[$4.a.5$]{GH} implies that $I(l_1 \cup l_2)$ is contained in $\oplus_k E_k$. 
  Now, \cite[Proposition $4$]{c1} and the argument after \cite[Proposition $5$]{c1} tells us that $\gamma'=a_1[l_1]+a_2[l_2]+b[H_X]$, where $\gamma'$ is the class $[C']$ of a line or a conic
  and $a_1, a_2, a_3$ are rationals. 
  
  By the above argument, $C'$ is not $l_1, l_2$ or $l_1 \cup l_2$. Denote by $t_0$ and $t_1$ the integers $l_1.C'$ and $l_2.C'$, respectively. 
  Then, intersecting the equality $[C']=a_1[l_1]+a_2[l_2]+b[H_X]$ by $H_X, l_1, l_2$ and $C'$, respectively, we have
  \begin{eqnarray}
  \deg(C')&=&a_1+a_2+bd\\
  t_0&=&a_1(2-d)+a_2+b\\
  t_1&=&a_1+a_2(2-d)+b\\
  C'^2&=&a_1t_0+a_2t_1+\deg(C')b
  \end{eqnarray}
  Assume that $l_i\not| C'$. Using adjunction formula one can check $C'^2$ is $2-d$ if $C'$ is a line and $6-2d$ if $C'$ is a conic.
  Then, using any standard programming language (for eg. Maple), one can see that there does not exist a solution to these set of equations.
  
  The only case that remains is when $C'$ is a conic and of the form $l' \cup l_i$ for $i=1$ or $2$ and $l'$ is distinct from $l_1, l_2$. We can then replace in the above equation $C'$ by $l'$ and replace 
  $a_1$ (resp. $a_2$) by $a_1-1$ (resp. $a_2-1$) if $i=1$ (resp. $i=2$). Then the above result tells us again that there are no solutions. So, $\codim \ov{\NL(\gamma)}$ has to be $2d-6$.
    In particular, $\ov{\NL(\gamma)}_{\red}$ equals $\ov{(\NL([l_1]) \cap \NL([l_2]))}_{\red}$.
   Using Theorem \ref{cas12}, we have that $h^1(\mo_X(-C)(d))=0$,
    Finally, Proposition \ref{a71} implies $\ov{\NL(\gamma)}$ is generically non-reduced. This completes the proof of the theorem.

  \end{proof}

\section{Generically non-reduced components of the Hilbert scheme of smooth curves}\label{mg}

  \begin{para}
     In this section we generalize the classical example of Mumford, to produce for any $d \ge 5$, new examples
     of \emph{generically non-reduced components} of the Hilbert schemes of smooth curves in $\p3$ contained in smooth degree $d$ surfaces of Picard number at least $3$ and 
     not contained in any surface of smaller degree (see Theorem \ref{phe3}). We use techniques developed in the previous sections.
     
     For an integer $d \ge 5$, we first give criterion under which a general element of the Hilbert scheme of curves
     is contained in a smooth degree $d$ surface (see Proposition \ref{mumsur21}). Then, we observe, if additionally
     the Hodge locus corresponding to a general element of this component is generically non-reduced, then so is the 
     component of the Hilbert scheme (see Theorem \ref{mumsur01}). We combine these results with Theorem \ref{a73} to 
     produce new explicit examples of non-reduced components of Hilbert schemes of smooth curves (see proof of Theorem \ref{phe3}).
  \end{para}
  
  \begin{prop}\label{mumsur21}
 Let $P_0$ be the Hilbert polynomial of a curve $C$ in $\p3$. Assume that there exists an integer $d$ and a  smooth degree $d$ surface, say $X$ containing $C$, such that $h^1(\mo_X(-C)(d))=0=h^0(\mo_X(-C)(d))$.
  Let $L$ be an irreducible component of $H_{P_0}$ containing $C$. Then, for a general element $D \in L$, $h^0(\I_{D}(d))>0$. Moreover, a general element $D \in L$ is contained in a smooth degree $d$ surface.
  \end{prop}
 
  \begin{proof}
   Since $X$ is a degree $d$ hypersurface, $\I_X \cong \mo_{\p3}(-d)$, hence
   $h^i(\I_X(d))=0$ for $i>1$. Since $h^1(\mo_X(-C)(d))=0$ and $h^3(\mo_X(-C))=0$ (Grothendieck's vanishing theorem), using the short exact sequence,
   \[0 \to \I_X(d) \to \I_C(d) \to \mo_X(-C)(d) \to 0\]
   we observe
   that $h^1(\I_C(d))=0$ and $h^3(\I_C(d))=0$. By the upper semi-continuity theorem
   there exists an open neighbourhood $U \subset H_{P_0}$ containing $C$ such that for all $u \in U$, the corresponding curve
   $\mc{C}_u$ satisfies, $h^1(\I_{\mc{C}_u}(d))=0=h^3(\I_{\mc{C}_u}(d))$,  hence, the Euler characteristic $\chi(\I_{\mc{C}_u}(d))=h^0(\I_{\mc{C}_u}(d))+h^2(\I_{\mc{C}_u}(d))$.
   
    Using upper semi-continuity theorem again, there exists an open subset $V \subset U$ containing $C$ such that for all $v \in V$,
    $h^2(\I_{\mc{C}_v}(d)) \le h^2(\I_{C}(d))$ and $h^0(\I_{\mc{C}_v}(d)) \le h^0(\I_{C}(d))$. Since the Euler characteristic is constant in families, we have \[h^0(\I_C(d))\ge h^0(\I_{\mc{C}_v}(d))=\chi(\I_{\mc{C}_v}(d))-h^2(\I_{\mc{C}_v}(d)) \ge \chi(\I_{C}(d))-h^2(\I_{C}(d))=h^0(\I_C(d)).\]
    Therefore, $h^0(\I_C(d))=h^0(\I_{\mc{C}_v}(d))$ for all $v \in V$. Since $h^0(\I_{C}(d))>0$, we conclude for all $v \in V$, $h^0(\I_{\mc{C}_v}(d))>0$.
        
   We now show a general element in $L$ can be embedded into a \emph{smooth} degree $d$-surface. 
   Since the fiber to the projection morphism $\pr_2:H_{P_0,Q_d} \to H_{P_0}$ over any point $D \in H_{P_0}$ is isomorphic
   $\mb{P}(I_d(D))$, one can check that there exists an irreducible component, say $H$ of $H_{P_0,Q_d}$ mapping 
   surjectively to $L$ whose generic fiber is isomorphic to $\mb{P}(I_d(C_g))$ for $C_g \in L$, generic.
   By the upper-semicontinuity of fiber dimension, the dimension of the fiber over $C$ to the morphism $\pr_2|_H$ is 
   of dimension at least $h^0(\I_d(C_g))-1$, which by the above argument forces it to be $h^0(\I_C(d))-1$. Since 
   the fiber $\pr_2^{-1}(C)$ is irreducible of dimension $h^0(\I_C(d))-1$, this means $\pr_2|_H^{-1}(C)=\mb{P}(I_d(C))$.
   Since $C$ is contained in a \emph{smooth} degree $d$ surface and a general element of $L$ is contained in a degree $d$ surface, it follows that a general element of $L$ is contained in a \emph{smooth} degree $d$ surface (use \cite[III. Ex. $10.2$]{R1}).
   The proposition then follows.
   \end{proof}

We recall a result due to Bloch which gives a relation between the Hodge locus corresponding to the cohomology class of a semi-regular curve and the corresponding flag Hilbert scheme.
\begin{thm}[{\cite[Theorem $7.1$]{b1}}]\label{b71}
 Let $X \to S$ be a smooth projective morphism with $S=\Spec(\mb{C}[[t_1,...,t_r]])$. Let $X_0 \subset X$ be the
 closed fiber and let $Z_0 \subset X_0$ be a local complete intersection subscheme of codimension $1$. Suppose that the
 topological cycle class $[Z_0] \in H^2(X_0,\mb{Z})$ lifts to a horizontal class $z \in F^1H^2(X)$ (where $F^\bullet$
 refers to the Hodge filtration) and that $Z_0$ is semi-regular in $X_0$. Then, $Z_0$ lifts to a local complete intersection subscheme $Z \subset X$.
\end{thm}

This result implies the following:

\begin{thm}\label{dim1}
 Let $X$ be a surface, $C$ be a semi-regular curve in $X$ and $\gamma \in H^{1,1}(X,\mb{Z})$ be the class of  $C$. 
 For any irreducible component $L'$ of $\overline{\NL(\gamma)}$
 there exists an irreducible component $H'$ of $H_{P,Q_{d_{\red}}}$ containing the pair $(C,X)$ such that  
 $\pr_2(H')$  coincides with $L'_{\red}$, where $\pr_2$ is  the second projection map from $H_{P,Q_d}$ to $H_{Q_d}$. 
\end{thm}

\begin{proof}
Using basic deformation theory, one can check that the image under $\pr_2:H_{P,Q_d} \to H_{Q_d}$ of all the irreducible components of $H_{{P,Q_d}_{\red}}$ containing the pair $(C,X)$
is contained in $\ov{\NL(\gamma)}$. But, Theorem \ref{b71} proves the converse i.e., $\ov{\NL(\gamma)}_{\red}$ is contained in 
$\pr_2(H_{P,Q_d})$. This proves the theorem.
\end{proof}

 The following result will help  us determine when a curve is semi-regular.

\begin{lem}\label{hf11}
Let $C$ be a connected reduced curve and $X$ a smooth degree $d$ surface containing $C$. Then, $H^1(\mo_X(-C)(k))=0$ for all $k \ge \deg(C)$. In particular, of $d \ge \deg(C)+4$ then $h^1(\mathcal{O}_X(C))=0$, hence $C$ is semi-regular.
\end{lem}

\begin{proof}
Since $X$ is a hypersurface in $\p3$ of degree $d$, $\I_X \cong \mo_{\p3}(-d)$.  Consider the short exact sequence:
\[0 \to \I_X \to \I_C \to \mo_X(-C) \to 0.\]
We get the following terms in the associated long exact sequence:
\[... \to H^1(\I_C(k)) \to H^1(\mo_X(-C)(k)) \to H^2(\I_X(k)) \to ...\]
Now, $H^2(\mo_{\p3}(k-d))=0$ (see \cite[Theorem $5.1$]{R1}). Now, $\I_{C}$ is $\deg(C)$-regular (see 
\cite[Main Theorem]{gi1}). 
So, $H^1(\I_C(k))=0$ for $k \ge \deg(C)$. This implies $H^1(\mo_X(-C)(k))=0$ for $k \ge \deg(C)$. 
By Serre duality, $0=H^1(\mo_X(-C)(d-4)) \cong H^1(\mo_X(C))$. So, $C$ is semi-regular.
\end{proof}

Using these results we can show the following theorem: 
   \begin{thm}\label{mumsur01}
   Let $X$ be a smooth degree $d$ surface, $\gamma \in H^{1,1}(X,\mb{Z})$ and $C$ be a smooth semi-regular curve in $X$. Assume that,
   $\gamma-[C]$ is a multiple of the class of the hyperplane section $H_X$, where $[C]$ is the cohomology class of $C$. Let $P_0$ be the Hilbert polynomial of $C$.
   If $h^1(\mo_X(-C)(d))=0=h^0(\mo_X(-C)(d))$ and $\ov{\NL(\gamma)}$ is irreducible, generically non-reduced then
      there is an irreducible generically non-reduced component of $H_{P_0}$ containing $C$ parametrizing smooth curves in $\p3$
      contained in a smooth surface of degree $d$ but not in any surface of smaller degree.
     \end{thm}

     \begin{proof}
     Since $\gamma-[C]$ is a multiple of the hyperplane section $H_X$, $\ov{\NL(\gamma)}$ is (scheme-theoretically) isomorphic
     to $\ov{\NL([C])}$. Hence, $\ov{\NL([C])}$ is generically non-reduced. 
     
     Since $C$ is semi-regular, Theorem \ref{dim1} implies that there exists an irreducible component $H_\gamma$ of $H_{{P_0,Q_d}_{\red}}$ such that $\pr_2(H_{\gamma})$ is isomorphic to $\NL([C])_{\red}$. 
     Denote by $L_\gamma:=\pr_1(H_{\gamma})$. By Lemma \ref{dim}, \[\dim \ov{\NL([C])}=\dim L_\gamma+\dim I_d(C_g)-h^0(\mo_{X_g}(C_g))\]for $(C_g,X_g) \in H_\gamma$, general.
     Now, Corollary \ref{hf12c} implies \[\dim T_{X_g}\NL([C])=\dim \pr_2T_{(C_g,X_g)}H_{P_0,Q_d}=\dim T_{(C_g,X_g)}H_{P_0,Q_d}-\dim \ker \beta_{C_g}.\]
     As $h^1(\mo_X(-C)(d))=0$, by the upper-semicontinuity theorem, for a general $(C_g,X_g) \in H_\gamma$, $h^1(\mo_{X_g}(-C_g)(d))=0$. Hence, Corollary \ref{dim4}
     implies $\pr_1T_{(C_g,X_g)}H_{P_0,Q_d} \cong H^0(\N_{C_g|\p3})$. By the normal short exact sequence, $\ker \beta_{C_g}=h^0(\N_{C_g|X_g})$. Therefore,
     \[\dim T_{X_g}\NL([C])=\dim T_{(C_g,X_g)}H_{P_0,Q_d}-h^0(\N_{C_g|X_g})=\]\[=\dim \pr_1T_{(C_g,X_g)}H_{P_0,Q_d}+\ker \rho_{C_g}-h^0(\N_{C_g|X_g})\]
     which is equal to $h^0(\N_{C_g|\p3})+h^0(\mo_{X_g}(-C_g)(d))-h^0(\N_{C_g|X_g})$. Using Lemma \ref{dim3}, we get
     \[\dim T_{X_g}(\NL([C]))-\dim \ov{\NL([C])}=h^0(\N_{C_g|\p3})-\dim L_\gamma.\] This implies $h^0(\N_{C_g|\p3})>\dim L_\gamma$ for a general $C_g \in L_\gamma$ because $\NL([C])$ is generically non-reduced.
     Since $L_\gamma$ is an irreducible component of $(H_{P_0})_{\red}$, the corresponding component of $H_{P_0}$, say $L'$
     is generically non-reduced. Since $C$ is smooth, a curve corresponding to a general closed point on $L'$ is smooth (see \cite[Ex. III.$10.2$]{R1}).
     
     It just remains to prove that a general element of $L'$ is not contained in a surface of degree smaller than $d$.
     For a general element $C_g \in L'$ contained in a smooth degree $d$ surface $X_g$, we have the following
     short exact sequence for all $k$:
     \[0 \to \I_{X_g}(k) \to \I_{C_g}(k) \to \mo_{X_g}(-C_g)(k) \to 0.\]
     For $k<d$, $h^i(\I_{X_g}(k))=h^i(\mo_{\p3}(k-d))=0$ for $i=0,1$. Hence, $h^0(\I_{C_g}(k))=h^0(\mo_{X_g}(-C_g)(k))$ for 
     $k<d$. Consider now the short exact sequence,
     \[0 \to \mo_{X_g}(-C_g)(k-1) \xrightarrow{p} \mo_{X_g}(-C_g)(k) \to \coker(p) \to 0.\]
     Since $h^0(\mo_{X_g}(-C_g)(d))=0$ one can recursively prove that $h^0(\mo_{X_g}(-C_g)(k))=0$ for any $k<d$.
     Therefore, $h^0(\I_{C_g}(k))=0$ for any $k<d$. This completes the proof of the theorem.
          \end{proof}

     The following technical lemma will be used in the proof of Theorem \ref{phe3}.
     \begin{lem}\label{mumsur05}
      Notations as in Theorem \ref{a73}. Let $d \ge 5$. The Castelnuovo-Mumford regularity of $\mo_X(2l_1+l_2)$ is at most $2d-5$.
     \end{lem}
    
 \begin{proof}
  Consider the short exact sequence, \[0 \to \mo_X(-l_1-l_2) \to \mo_X \to \mo_{l_1+l_2} \to 0.\]
     \cite[Ex. III.$5.5$]{R1} implies that for all $k \in \mb{Z}$, the induced map $H^0(\mo_X(k)) \to H^0(\mo_{l_1+l_2}(k))$ is surjective and $H^1(\mo_X(k))=0$.
     So, \[0=H^1(\mo_X(-l_1-l_2)(k))\stackrel{{\small \mbox{SD}}}{=}H^1(\mo_X(l_1+l_2)(d-4-k))^\vee, \, \forall k \in \mb{Z}.\]
     In other words, $H^1(\mo_X(l_1+l_2)(-k+d-4))=0$ for all $k \in \mb{Z}$. Now, $H^2(\mo_X(l_1+l_2)(k))\stackrel{{\small \mbox{SD}}}{=}
     H^0(\mo_X(-l_1-l_2)(-k+d-4))$ is zero if the degree of $\mo_X(-l_1-l_2)(-k+d-4)$ is less than zero, which happens if $k>d-6$. 
     Hence, the Castelnuovo-Mumford regularity of $\mo_X(l_1+l_2)$ is at most $d-4$.
     
      Using Serre duality, we can conclude \[H^1(\mo_{l_1} \otimes \mo_X(2l_1+l_2)(k))^\vee=H^0(\mo_{l_1} \otimes \mo_X(-2l_1-l_2)(-k)(-2)).\]
     Now, $\deg(\mo_{l_1} \otimes \mo_X(-2l_1-l_2)(-k)(-2))=l_1(-2l_1-l_2-(k+2)H_X)=-2(2-d)-1-(k+2)=2d-7-k$ where the second last equality 
     follows from $l_1^2=2-d$ which can be computed using the adjunction formula. Therefore, for $k>2d-7$, $H^1(\mo_{l_1} \otimes \mo_X(2l_1+l_2)(k))=0$,
     whence the Castelnuovo-Mumford regularity of $\mo_{l_1} \otimes \mo_X(2l_1+l_2)$ is at most $2d-5$.
     
     Finally, using the short exact sequence:\[0 \to \mo_X(l_1+l_2) \to \mo_X(2l_1+l_2) \to \mo_{l_1} \otimes \mo_X(2l_1+l_2) \to 0\]
     we observe that the Castelnuovo-Mumford regularity of $\mo_X(2l_1+l_2)$ is at most $\max\{2d-5,d-4\}$, which is 
     $2d-5$ as $d \ge 5$. This completes the proof of the lemma.
 \end{proof}

     We finally arrive at the proof of the final theorem of the article.

 \begin{proof}[Proof of Theorem \ref{phe3}]
   Let $\gamma'$ be the cohomology class of $C$.
    Theorem \ref{a73} states that $\ov{\NL(\gamma')}$ is irreducible, generically non-reduced. 
    By Lemma \ref{mumsur05}, $\mo_X(2l_1+l_2)$ is $2d-5$-regular.
     Hence, $\mo_X(C)(m-1)$ is globally generated for $m \ge 2d-4$ and a general element of the linear system 
     $|\mo_X(C)(m)|$ is semi-regular. Then, \cite[Ex. II. $7.5$(d)]{R1}
  states that $\mo_X(C)(m)$ is very ample. Therefore, Bertini's theorem implies that a general curve $C'$ in the linear system corresponding to $\mo_X(C)(m)$ 
  is smooth and semi-regular for $m \ge 2d-4$. 
  
  Furthermore, $h^1(\mo_X(-C')(d))=h^1(\mo_X(-C)(d-m))=h^1(\mo_X(C)(m-4))=0$ as $m-4 \ge 2d-6$. 
  Since $\deg(\mo_X(-C)(d-m))<0$ for  $m \ge 2d-2$, $h^0(\mo_X(-C')(d))=h^0(\mo_X(-C)(d-m))=0$.
   Denote by $P'$ the Hilbert polynomial of $C'$. 
    Then, Theorem \ref{mumsur01} implies that there exists an irreducible generically non-reduced component, say $L'$ of $H_{P'}$ containing $C'$ and parametrizing smooth curves
    contained in a smooth surface of degree $d$ but not in any surface of smaller degree.

    Clearly, $\deg(C')=md+3$ and by the adjunction formula, 
    \begin{eqnarray*}
     \rho_a(C')&=&1+(1/2)(C'^2+(d-4)\deg(C'))\\
      &=&1+(1/2)(C^2+m^2d+2m\deg(C)+(d-4)(md+3))\\
     &=&1+(1/2)(4l_1^2+l_2^2+4+md^2+d(m^2-4m+3)-12+6m)\\
     &=&1+(1/2)(4(2-d)+(2-d)+md^2+d(m^2-4m+3)-8+6m)\\
     &=&1+(1/2)(md^2+d(m^2-4m-2)+6m+2).
    \end{eqnarray*}
    It is easy to observe that in this case, $\NL(\gamma')$ parametrizes surfaces containing two coplanar lines, which are not (global) complete intersections in $X_g$.
    This directly implies the sublattice of the Picard lattice generated by the cohomology classes of the two coplanar lines and the hyperplane section is of rank $3$.
     This completes the proof of the theorem.
 \end{proof}

\end{document}